\documentclass[12pt]{amsart}
\usepackage{latexsym} 
  \usepackage[all]{xy}
  \usepackage{amsfonts} 
  \usepackage{amsthm} 
  \usepackage{amsmath} 
  \usepackage{amssymb}
  \usepackage{pifont}  
  \usepackage{enumerate}
  \usepackage{dcolumn}
  \usepackage{comment}
\usepackage{hyperref}  
\newcolumntype{2}{D{.}{}{2.0}}
\usepackage{graphicx}
  \xyoption{2cell}

 \usepackage{tikz}
\usetikzlibrary{arrows} 
\numberwithin{equation}{section}

\def\sw#1{{\sb{(#1)}}}

    \def\<{{\langle}}
    \def\>{{\rangle}}
    
    \def\eps{\varepsilon}

    \def\note#1{{}}
    \def\can{{\rm can}}
    
    \def\note#1{}

       \def\Tnr#1{{\mathrm{Tn}(#1)}}
       \def\Tnl#1{{\widehat{\mathrm{Tn}}(#1)}}

    \def\can{{\rm can}}

    \def\beq{\begin{equation}}
    \def\eeq{\end{equation}}

    \def\id{\mathrm{id}}

    \def\ot{{\otimes}}

    \def\cten#1{\raise-.2cm\hbox{$\stackrel{\displaystyle\square}
{\scriptscriptstyle{#1}}$}}

     \def\FF{\mathbb{F}}

     \newcounter{zlist}
  \newenvironment{zlist}{\begin{list}{(\arabic{zlist})}{
  \usecounter{zlist}\leftmargin2.5em\labelwidth2em\labelsep0.5em
  \topsep0.6ex
  \parsep0.3ex plus0.2ex minus0.1ex}}{\end{list}}

  \newcounter{blist}
  \newenvironment{blist}{\begin{list}{(\alph{blist})}{
  \usecounter{blist}\leftmargin2.5em\labelwidth2em\labelsep0.5em
  \topsep0.6ex 
  \parsep0.3ex plus0.2ex minus0.1ex}}{\end{list}}

\newtheorem{proposition}{Proposition}[section]
\newtheorem{lemma}[proposition]{Lemma}

\newtheorem{corollary}[proposition]{Corollary}
\newtheorem{theorem}[proposition]{Theorem}
\theoremstyle{definition}
\newtheorem{definition}[proposition]{Definition}
\newtheorem{example}[proposition]{Example}

\theoremstyle{remark} 
\newtheorem{remark}[proposition]{Remark}

\begin{document}
\title[{Hopf heaps}]{Translation Hopf algebras and Hopf heaps}

   \author[Brzezi\'nski]{Tomasz Brzezi\'nski}
\address{
Department of Mathematics, Swansea University, 
Swansea University Bay Campus,
Fabian Way,
Swansea,
  Swansea SA1 8EN, U.K.\ \newline \indent
Faculty of Mathematics, University of Bia{\l}ystok, K.\ Cio{\l}kowskiego  1M,
15-245 Bia\-{\l}ys\-tok, Poland}
\email{T.Brzezinski@swansea.ac.uk}

   \author[Hryniewicka]{Ma\l gorzata Hryniewicka}
\address{
Faculty of Mathematics, University of Bia{\l}ystok, K.\ Cio{\l}kowskiego  1M,
15-245 Bia\-{\l}ys\-tok, Poland}
\email{margitt@math.uwb.edu.pl}

     \date{\today}
   \begin{abstract}
To every Hopf heap or quantum cotorsor of Grunspan a Hopf algebra  of translations is associated. This translation Hopf algebra acts on the Hopf heap making it a Hopf-Galois co-object. Conversely, any Hopf-Galois co-object has the natural structure of a Hopf heap with the translation Hopf algebra isomorphic to the acting Hopf algebra. It is then shown that this assignment establishes an equivalence  between categories of Hopf heaps and Hopf-Galois co-objects. 
   \end{abstract}
   \maketitle

\section{Introduction}
Introduced in the 1920s by Pr\"ufer \cite{Pru:the} and Baer \cite{Bae:ein} {\em heaps} are simple algebraic systems comprising a set $X$ and a ternary operation $[-,-,-]$ on $X$. The axioms (see \eqref{heap} below) ensure that any non-empty heap can be retracted to a family of isomorphic groups, one for each element of $X$, and -- conversely -- any group can be given a heap operation by the suitable combination of the group binary operation and the inverses. The latter assignment constitutes a functor from the category of groups to that of heaps. In the opposite direction, one can functorially assign to a non-empty heap a group of translations, denoted $\Tnr X$, i.e.\ all maps $\tau_a^b: X\to X$, $c\mapsto [c,a,b]$, $a,b\in X$. The group $\Tnr X$ acts on $X$ freely and transitively, thus making $X$ into a $\Tnr X$-torsor. The functor $(X, [-,-,-]) \mapsto (\Tnr X, X)$ establishes an equivalence between the category of heaps and torsors (see  \cite{BreBrzRybSar:hea} for a recent discussion). 

This note is concerned with the linearisation of heaps proposed by Grunspan in \cite{Gru:tor},  termed {\em quantum cotorsors} there and referred to as {\em Hopf heaps} 
in the present text. Adopting the results of \cite{Sch:tor} and \cite{Gru:tor} (see also \cite{Sch:Hop}) we assign to each Hopf heap $C$ 
two Hopf algebras $\Tnr{C}$ and $\Tnl C$ that act on $C$ turning it into a bimodule coalgebra and make $C$ into a bi-Galois co-object (a notion dual to that of a bi-Galois object introduced in \cite{Sch:big})). This assignment establishes an equivalence between the categories of Hopf heaps and bi-Galois co-objects and also, dually to \cite{Sch:few}  gives a construction and thus the proof of the existence of the {\em Grunspan map}, which was assumed as a part of the original definition of a quantum cotorsor. 

The main novelty of this paper does not reside in bringing the results of Grunspan \cite{Gru:tor} and Schauenburg \cite{Sch:tor} and \cite{Sch:few} to the dual situation, which rightly in our opinion might be considered as a formulaic exercise, but rather in giving an alternative description of the correspondence between bi-Galois co-objects and Hopf heaps which does not seem to be available in the original setup of quantum torsors. This characterisation in terms of linear endomorphisms of $C$ is similar to the functor assigning the group of translations to a heap evoked earlier, and thus closer to that encountered in the classical geometric or set-theoretic set-up.

We work over a field $\FF$. All coalgebras, typically denoted by $C$ (or $H$ if a Hopf algebra) are over $\FF$, coassociative,  counital and of dimension at least one. The coproduct in $C$ is denoted by $\Delta$ and counit by $\eps$. We use the Sweedler notation to denote the coproducts in the form $\Delta(c) = \sum c_\sw 1\otimes c_\sw 2$, $(\Delta\ot \id)\circ\Delta(c) = \sum c\sw 1\ot c\sw 2\ot c\sw 3$, etc. The coalgebra co-opposite to $C$, i.e.\ with the comultiplication $c\mapsto \sum c_\sw 2\otimes c_\sw 1$ is denoted by $C^{\mathrm{co}}$. The set of group-like elements of $C$ is recorded as $\mathrm{G}(C)$. All algebras are associative and with identity. The algebra opposite to $A$ is denoted  by $A^\mathrm{op}$.  In any Hopf algebra $S$ stands for the antipode.

\section{Hopf heaps and translation Hopf algebra}\label{sec.heaps}
A {\em heap} is a set $X$ together with a ternary operation $[-,-,-]:X^3\to X$ such that, for all $x_1,\ldots, x_5\in X$,
\begin{equation}\label{heap}
[x_1,x_2,[x_3,x_4,x_5]] = [[x_1,x_2,x_3],x_4,x_5], \quad [x_1, x_1, x_2] = x_2, \quad [x_1, x_2, x_2] = x_1.
\end{equation}
The category of sets is a monoidal category with the monoidal product given by the Cartesian product and the singleton set as the monoidal unit. Every set is then a comonoid (coalgebra) in the unique way with the comultiplication given by the diagonal map $x\mapsto (x,x)$ and the counit the unique map from $X$ to the (fixed) singleton set. Both these maps clearly feature in the second and third equations \eqref{heap}. Extending the definition of a heap to the monoidal category of vector spaces one thus needs to consider a general coassociative and counital coalgebra  as the underlying object and use comultiplication and counit as appropriate replacements in \eqref{heap}. This leads to the following definition which is dual to that of a quantum torsor in \cite{Gru:tor} or quantum heap in \cite{Sko:hea}.
\begin{definition}\label{def.heap}
 A {\em Hopf heap} 
 is a coalgebra $C$ together with a coalgebra map
 $$
 \chi: C\otimes C^{\mathrm{co}}\otimes C\to C, \qquad a\otimes b\otimes c \mapsto [a,b,c],
 $$
 such that for all $a,b, c, d,e \in C$,
 \begin{subequations}\label{l.heap}
     \begin{equation}\label{l.heap.a}
 [[a,b,c],d,e]=[a,b,[c,d,e]],\\
     \end{equation}
     \begin{equation}\label{l.heap.c}
 \sum [c\sw 1, c\sw 2,a] =   \sum [a,c\sw 1,c\sw 2] = \eps(c)a.      
     \end{equation}
 \end{subequations}
A morphism of Hopf heaps $(C,\chi_C)$ and $(D,\chi_D)$ is a coalgebra map $f:C\to D$ rendering commutative the following diagram
\begin{equation}\label{l.heap.mor}
 \xymatrix{C\ot C^{\mathrm{co}}\ot C \ar[rr]^-{\chi_C}\ar[d]_{f\ot f\ot f} & & C\ar[d]^f\\
D\ot D^{\mathrm{co}}\ot D\ar[rr]^-{\chi_D} & & D;}
   \end{equation}
on elements,
\begin{equation}\label{heap.mor.el}
    f([a,b,c]) = [f(a),f(b),f(c)], \qquad \mbox{for all} \; a,b,c\in C.
\end{equation}

A {\em Grunspan map} for a Hopf heap $(C,\chi_C)$ is a coalgebra homomorphism $\vartheta: C\to C$, such that, for all $a,b,c,d,e\in C$,
\begin{equation}\label{assoc.theta}
[[a,b,\vartheta(c)],d,e] = [a,[d,c,b],e].
\end{equation} 
The category of Hopf heaps (over the fixed field $\FF$) is denoted by $\mathcal{HH}$.
\end{definition}

\begin{remark}\label{rem.map}
One can easily calculate that, if it exists, the Grunspan map for a Hopf heap $(C,\chi_C)$ is given by the formula
\begin{equation}\label{Grunspan.map}
  \vartheta : C\to C, \qquad c\mapsto \sum[c\sw 1, [c\sw 4,c\sw 3,c\sw 2], c\sw 5],
  \end{equation} 
  and thus necessarily is unique.

In fact, parallel to the situation described in  \cite{Sch:few}, the forthcoming results will show that  a Hopf heap always admits the (unique) Grunspan map (see Corollary~\ref{cor.Grunspan}).

The formula \eqref{Grunspan.map} together with the coalgebra map property of homomorphisms of Hopf heaps and \eqref{heap.mor.el} ensure that homomorphisms commute with Grunspan maps, that is, if $f:C\to D$ is a homomorphism of Hopf heaps with respective Grunspan maps $\vartheta_C$ and $\vartheta_D$, then 
\begin{equation}\label{mor.theta}
f\circ \vartheta_C = \vartheta_D\circ f.
    \end{equation}
\end{remark}

\begin{example}\label{ex.heap}
 If $H$ is a Hopf algebra, then $H$ is a Hopf heap with the operation $[a,b,c] = aS(b)c$. The Grunspan map is then the square of the antipode, i.e.\ $\vartheta = S\circ S$.
 
 Conversely, given a Hopf heap $(C,\chi)$, for any $x\in \mathrm{G}(C)$, the coalgebra $C$ is made into a Hopf algebra with  identity $x$, and multiplication and antipode,
 $$
 ab = [a,x,b], \qquad S(a) = [x,a,x].
 $$
 This Hopf algebra is denoted by $\mathrm{H}_x(C)$.
 One easily checks that the Hopf heap associated to the Hopf algebra $\mathrm{H}_x(C)$ is equal to $C$.
 
 These examples mimic the standard correspondence between groups and heaps.
\end{example}

The key object analysed in this paper is introduced in the following definition.
\begin{definition}
    Let $(C,\chi)$ be a Hopf heap. For all $a,b\in C$, the linear map 
    $$
    \tau_a^b: C\to C, \qquad c\mapsto \chi(c\ot a\ot b) = [c,a,b],
    $$
    is called a {\em right $(a,b)$-translation}. The space spanned by all right $(a,b)$-translations is denoted by $\Tnr C$, that is,
    $$
    \Tnr C := \FF\langle \tau_a^b\;|\; a,b\in C\rangle. 
    $$

    Symmetrically, linear maps 
    $$
    \sigma^a_b: C\to C, \qquad c\mapsto \chi(a\ot b\ot c) = [a,b,c],
    $$
    are called {\em left $(a,b)$-translations} and the space spanned by all of them is denoted by $\Tnl C$.
\end{definition}

In what follows we will concentrate on right translations, the corresponding results for left translations (of which we mention briefly in summary) are obtained by symmetric arguments. 
The following lemma gathers basic properties of $(a,b)$-translations.

\begin{lemma}\label{lem.tau}
Let $(C,\chi)$ be a Hopf heap. Then, for all $a,b,c,d \in C$,
\begin{subequations}
\begin{equation}\label{tau.coal}
        \Delta(\tau_a^b(c)) = \sum \tau_{a\sw 2}^{b\sw 1}(c\sw 1)\ot \tau_{a\sw 1}^{b\sw 2}(c\sw 2),
\end{equation}
\begin{equation}\label{tau.theta}
    \sum \tau_{a\sw 1}^{[a\sw 2,b,c]}  = \eps(a)\tau_b^c,
\end{equation}
\begin{equation}\label{tau.theta.a}
    \sum \tau_{a\sw 1}^{a\sw 2} = 
    \eps(a)\id ,
\end{equation}
\begin{equation}\label{tau.comp}
    \tau_c^d\circ \tau_a^b = \tau_a^{[b,c,d]}.
\end{equation}
\end{subequations}
In addition if the Grunspan map $\vartheta$ exists, then
\begin{subequations}
\begin{equation}\label{tau.theta.1}
     \sum \tau_{a\sw 2}^{[\vartheta(a\sw 1),b,c]} = \eps(a)\tau_b^c,
\end{equation}
\begin{equation}\label{tau.theta.a.1}
     \sum \tau_{a\sw 2}^{\vartheta(a\sw 1)} = \eps(a)\id ,
\end{equation}
\begin{equation}\label{tau.comp.theta}
    \tau_c^d\circ \tau_a^{\vartheta(b)} = \tau_{[c,b,a]}^d.
\end{equation}
\end{subequations}
\end{lemma}
\begin{proof}
    Equation \eqref{tau.coal} follows immediately from the fact that $\chi$ is a coalgebra map. To prove \eqref{tau.theta}, compute
    $$
    \begin{aligned}
     \sum \tau_{a\sw 1}^{[a\sw 2,b,c]}(d) &=  \sum [d, a\sw 1, [a\sw 2,b,c]] \\
     &= \sum [[d, a\sw 1, a\sw 2],b,c] = \eps(a) [d,b,c] = \eps(a) \tau_b^c(d),
    \end{aligned}
    $$
    by equations \eqref{l.heap}. In addition adopting \eqref{assoc.theta} we find
    $$
    \begin{aligned}
        \sum \tau_{a\sw 2}^{[\vartheta(a\sw 1),b,c]}(d) &= \sum [d,a\sw 2,[\vartheta(a\sw 1),b,c] ] 
        = \sum [[d,a\sw 2,\vartheta(a\sw 1)],b,c] \\
        &= \sum [d,[b, a\sw 1, a\sw 2],c] = \eps(a) [d,b,c] = \eps(a) \tau_b^c(d),
    \end{aligned}
    $$
    which proves \eqref{tau.theta.1}.

    Equations \eqref{tau.theta.a} and \eqref{tau.theta.a.1} follow from \eqref{tau.theta} and \eqref{tau.theta.1}, since, first by \eqref{l.heap.c} 
    $$
    \sum \tau_{a\sw 1}^{a\sw 2} = \eps(a)\id,
    $$
    and thus, second, 
    $$
    \eps(a) \id = \sum \tau_{a\sw 1}^{a\sw 2} = \sum \tau_{a\sw 1}^{[a\sw 2,a\sw 3,a\sw 4]} = \sum \tau_{a\sw 2}^{[\vartheta(a\sw 1),a\sw 3,a\sw 4]} = \sum \tau_{a\sw 2}^{\vartheta(a\sw 1)},
    $$
    by \eqref{tau.theta} and \eqref{tau.theta.1}, and  \eqref{l.heap.c} again.

    Finally, equations \eqref{tau.comp}  and \eqref{tau.comp.theta} follow by \eqref{l.heap.a} and \eqref{assoc.theta}.
\end{proof}

Equation \eqref{tau.comp} in Lemma~\ref{lem.tau} implies in particular that $\Tnr C$ is closed under the composition. Furthermore, since any non-zero coalgebra over a field has at least one element with a non-zero counit, \eqref{tau.theta.a} shows that $\id \in \Tnr C$.

\begin{theorem}\label{thm.trans}
Let $(C,\chi)$ be a Hopf heap.
\begin{zlist}
 \item   The space $\Tnr C$ is a bialgebra  with multiplication given by the opposite composition, and comultiplication $\Delta$ and counit $\eps$:
     \begin{equation}\label{trans.com.cou}
     \Delta(\tau_a^b) = \sum \tau_{a\sw 2}^{b\sw 1}\ot \tau_{a\sw 1}^{b\sw 2}, \qquad \eps(\tau_a^b) = \eps(a)\eps(b),  
     \end{equation}
for all $a,b\in C$.
\item If  $(C,\chi)$ admits the Grunspan map $\vartheta$, then $\Tnr C$ is a Hopf algebra with the antipode  
\begin{equation}\label{trans.ant}
      S(\tau_a^b) = \tau_{b}^{\vartheta(a)},    \end{equation}
      for all $a,b\in C$.
\item If $f:C\to D$ is a morphism of Hopf heaps, then the function
\begin{equation}\label{trans.mor}
  \Tnr f : \Tnr C\to \Tnr D, \qquad \tau_a^b \mapsto \tau_{f(a)}^{f(b)},  
\end{equation}
is a bialgebra map, hence a Hopf algebra homomorphism whenever the Grunspan map exists.
\item The assignment $C\mapsto \Tnr C$, $f\mapsto \Tnr f$ defines a functor from the category of Hopf heaps (with Grunspan maps) to the category of bialgebras (resp.\ Hopf algebras).
\end{zlist}
\end{theorem}
\begin{proof}
    (1) In view of the composition property \eqref{tau.comp}, the multiplication in $\Tnr C$, denoted by juxtaposition  comes out as
    \begin{equation}\label{prod.Hopf}
        \tau_a^b\tau_c^d = \tau_a^{[b,c,d]}, \qquad \mbox{for all}\;\; a,b,c,d\in C.
    \end{equation}
    The coassociativity and comultiplicativity of $\Delta$ and the counit property follow immediately from \eqref{trans.com.cou} and the fact that $\chi$ is a counital coalgebra homomorphism. The unitality of $\Delta$ is a consequence of \eqref{tau.theta.a}.
    
    (2) If the Grunspan map $\vartheta$ exists, then we can use  \eqref{tau.comp.theta} and \eqref{tau.theta.a.1} to obtain
    $$
    \sum S(\tau_{a\sw 2}^{b\sw 1}) \tau_{a\sw 1}^{b\sw 2} = \sum \tau_{a\sw 1}^{b\sw 2}\circ \tau_{b\sw 1}^{\vartheta(a\sw 2)}  = \sum \tau_{[a\sw 1, a\sw 2,b\sw 1]}^{b\sw 2} = \eps(a)\eps(b)\id
    $$
    and
    $$
    \sum \tau_{a\sw 2}^{b\sw 1} S(\tau_{a\sw 1}^{b\sw 2}) = \sum \tau_{b\sw 2}^{\vartheta(a\sw 1)}\circ \tau_{a\sw 2}^{b\sw 1}  = \sum \tau_{a\sw 2}^{[b\sw 1, b\sw 2,\vartheta(a\sw 1)]} = \eps(a)\eps(b)\id . 
    $$
    Therefore, $S$ is the antipode and $\Tnr C$ is a Hopf algebra as stated.

    (3) Since $f$ is a coalgebra map,
    $$
    \begin{aligned}
        \Delta\left(\Tnr f(\tau_a^b)\right) &= \sum \tau_{f(a)\sw 2}^{f(b)\sw 1} \ot \tau_{f(a)\sw 1}^{f(b)\sw 2} \\
        & = \sum \tau_{f(a\sw 2)}^{f(b\sw 1)} \ot \tau_{f(a\sw 1)}^{f(b\sw 2)} = (\Tnr f \ot \Tnr f) \circ\Delta \left(\tau_a^b\right),
    \end{aligned}
    $$
    and 
    $$
    \eps(\Tnr f(\tau_a^b)) = \eps(f(a))\eps(f(b)) = \eps(a)\eps(b) = \eps (\tau_a^b).
    $$
    Hence $\Tnr f$ is a coalgebra map. Again, by the coalgebra map property of $f$, for all $a\in C$, 
    $$
    \Tnr f (\eps(a) \id)= \sum \Tnr f\left(\tau_{a\sw 1}^{a\sw 2}\right) = \eps(a) \id,
    $$
    so, $\Tnr f(\id) = \id $. Combination of   \eqref{prod.Hopf} with \eqref{heap.mor.el} yields the multiplicativity of $\Tnr f$. Explicitly,
    $$
    \begin{aligned}
    \Tnr f(\tau_a^b \tau_c^d) &= \Tnr f (\tau_a^{[b,c,d]}) = \tau_{f(a)}^{f([b,c,d])}\\ &= \tau_{f(a)}^{[f(b),f(c),f(d)]}
    = \tau_{f(a)}^{f(b)} \tau_{f(c)}^{f(d)} = \Tnr f(\tau_a^b)  \Tnr f(\tau_c^d).
    \end{aligned}
    $$
    Hence $\Tnr f$ is a bialgebra homomorphism.
    
    Finally, if the Grunspan map exists, then for all $a,b\in C$,
    $$
    \Tnr f (S(\tau_a^b)) = \Tnr f\left(\tau_b^{\vartheta (a)}\right) = \tau_{f(b)}^{f(\vartheta (a))} = \tau_{f(b)}^{\vartheta (f(a))}= S\left(\Tnr f \left(\tau_a^b\right)\right),
    $$
    where the penultimate equality follows by \eqref{mor.theta}. Therefore, $\Tnr f$ is a Hopf algebra map as stated.

    (4) The fact that $\Tnr \id = \id$ and the preservation of composition of morphisms are obvious. Hence $\mathrm{Tn}$ is a functor as claimed.
\end{proof}

\begin{remark}\label{rem.left}
By symmetric arguments, the space $\Tnl C$ of left $(a,b)$-translations of a Hopf heap $C$ with the Grunspan map $\vartheta$ is a Hopf algebra with operations, for all $\sigma^a_b,\sigma^c_d\in \Tnl C$,
\begin{subequations}\label{left.str}
    \begin{equation}\label{left.str.a}
        \sigma^a_b\sigma^c_d = \sigma^a_b\circ \sigma^c_d = \sigma^{[a,b,c]}_d,
    \end{equation}
    \begin{equation}\label{left.str.c}
        \Delta(\sigma^a_b) = \sum \sigma^{a\sw 1}_{b\sw 2}\ot \sigma^{a\sw 2}_{b\sw 1}, \qquad \eps (\sigma^a_b) = \eps(a)\eps(b), \qquad S(\sigma^a_b) = \sigma^{\vartheta(b)}_a.
    \end{equation}
\end{subequations}
The obvious coalgebra isomorphism $\Tnl C\to \Tnr C$, $\sigma^a_b\mapsto \tau_b^a$ is an isomorphism of Hopf algebras $\Tnl C^{\mathrm{op}}\cong \Tnr C$ in the abelian Hopf heap case only, that is if and only if, for all $a,b,c\in C$, $[a,b,c] = [c,b,a]$. Notwithstanding, similarly to the right translations case, the assignment 
$$
  \Tnl{-} : C\mapsto \Tnl C, \qquad 
  \left(
  \xymatrix{C\ar[r]^f & D}
  \right)
  \mapsto  
  \left(
  \xymatrix{\Tnl C\ar[r]^{\Tnl f} & \Tnl D}
  \;\; \sigma^a_b\mapsto \sigma^{f(a)}_{f(b)}\right),  
$$
is a functor from the category of Hopf heaps (with Grunspan maps) to the category of bialgebras (resp.\ Hopf algebras).
\end{remark}

\begin{definition}\label{def.tran.Hop}
For  a Hopf heap $(C,\chi)$,  $\Tnr C$ is called the {\em right translation Hopf algebra} and $\Tnl C$ is called the {\em left translation Hopf algebra}.
\end{definition}

\begin{remark}
     {\em A priori} $\Tnr C$ and $\Tnl C$ are simply bialgebras, however, in view of the forthcoming Corollary~\ref{cor.trans.hopf}, {\em a posteriori} both are Hopf algebras, thus justifying the terminology.
\end{remark}

\begin{proposition}\label{prop.grouplike}
Let $(C,\chi)$ be a Hopf heap. Then, for all $x\in \mathrm{G}(C)$, 
$$
\mathrm{H}_x(C) \cong \Tnr C\cong \Tnl C,
$$ 
as bialgebras. Consequently $\Tnr C$ and $\Tnl C$ are Hopf algebras.
\end{proposition}
\begin{proof}
    Let us consider the map
    \begin{equation}\label{iso.Hopf}
    \varphi: \mathrm{H}_x(C) \to \Tnr C, \qquad a \mapsto \tau_x^a.
    \end{equation}
    The map is a coalgebra homomorphism, since $x$ is a group-like element. Using \eqref{prod.Hopf} and \eqref{tau.theta.a} one immediately concludes that $\varphi$ is an algebra homomorphism. Equation \eqref{tau.theta} together with the definitions of the antipodes in $\mathrm{H}_x(C)$  and the right translation Hopf algebra $\Tnr C$ allow one to verify, for all $a\in C$
    $$
    \varphi(S(a)) = \tau_x^{S(a)} = \tau_x^{[x,a,x]}  = \tau_a^x.
    $$

    In the opposite direction we define the map
    $$
    \varphi^{-1}: \Tnr C \to \mathrm{H}_x(C), \qquad \tau_a^b\mapsto [x,a,b].
    $$
    Then, for all $a,b\in C$,
    $$
    \varphi\circ \varphi^{-1}(\tau_a^b) = \tau_x^{[x,a,b]} = \tau_x^x\tau_a^b = \eps(x)\tau_a^b = \tau_a^b,
    $$
    and
    $$
    \varphi^{-1}\circ \varphi(a) = [x,x,a] =\eps(x) a =a.
    $$
    Therefore, $\varphi^{-1}$ is the inverse of the bialgebra algebra map $\varphi$. 

    The isomorphism $H_x(C)\cong \Tnl C$ is given by $a\mapsto \sigma^a_x$.

    For the last assertion, since $H_x(C)$ is a Hopf algebra, its antipode $S$ can be exported to $\Tnr C$ and $\Tnl C$ via the respective bialgebra isomorphism. For example the antipode of $\Tnr C$ comes out as
    $$
    S(\tau_a^b) = \varphi\circ S\circ\varphi^{-1}(\tau_a^b) = \tau_{[x,a,b]}^x,
    $$
    for all $a,b\in C$.
\end{proof}

\section{Hopf heaps and Hopf-Galois co-objects}
Let $H$ be a Hopf algebra. Recall that a coalgebra $C$ is a {\em right $H$-module coalgebra} if $C$ is a right $H$-module, such that,  for all $h\in H$, $c\in C$,
\begin{equation}\label{mod.coa}
   \Delta(c\cdot h) = \sum c\sw 1\cdot h\sw 1\ot c\sw 2\cdot h\sw 2, \qquad \eps(c\cdot h) = \eps(c)\eps(h),
\end{equation}
 where the dot in-between elements denotes the action of $H$ on $C$.
A left $H$-module coalgebra is defined symmetrically. Similarly to Hopf-Galois objects and bi-Galois objects defined as Hopf-Galois extensions \cite{KreTak:Gal}, respectively  bi-Galois extensions \cite{Sch:big}, with trivial coinvariants Hopf-Galois co-objects are defined as Hopf-Galois co-extensions \cite[Section~4]{Sch:pri} with trivial invariants. 
\begin{definition}\label{def.Gal}
    A right $H$-module coalgebra $C$ is a {\em right Hopf-Galois co-object} if 
    \begin{blist}
        \item $\ker\eps = \FF\langle c\cdot h - c\eps(h) \;|\; c\in C, h\in H\rangle$,
        \item the {\em canonical map}
        \begin{equation}\label{can}
            \can: C\ot H \to C\ot C, \qquad c\ot h \mapsto \sum c\sw 1 \ot c\sw 2\cdot h,
        \end{equation}
        is an isomorphism.
    \end{blist}
    A {\em left Hopf-Galois co-object} is defined symmetrically. A coalgebra $C$ that is both a right and left Hopf-Galois co-object of Hopf algebras whose actions on $C$ commute (that is, $C$ is a bimodule coalgebra) is called a {\em bi-Galois co-object}.
\end{definition}

We note in passing that the notion of a bi-Galois co-object is secondary to that of a Hopf-Galois co-object, since, as shown in the dual set-up in \cite{Sch:big}, every (right) Hopf-Galois co-object yields a Hopf algebra making it into a bi-Galois co-object. This construction follows the Ehresmann association of a structural group or gauge groupoid to a principal bundle (see \cite{Pra:Ehr} for overview, historic background and references), and hence the resulting Hopf algebra is termed an {\em Ehresmann-Schauneburg} Hopf algebra. We outline this construction presently.

Let $C$ be a right $H$-Hopf-Galois co-object with the canonical isomorphism $\can$. The {\em cotranslation map} $\tau : C\ot C\to H$ is defined by the formula
\begin{equation}\label{cotrans}
    \tau = (\eps \ot \id)\circ \can^{-1}.
\end{equation}
The following properties of the cotranslation map  \eqref{cotrans} (see e.g.\ \cite[Section~34.17]{BrzWis:cor}) play a key role in what follows. For all $a,b\in C$, $h\in H$,
    \begin{subequations}
        \begin{equation}\label{trans.counit}
        \eps(\tau(a\ot b)) = \eps(a)\eps(b),
        \end{equation}
        \begin{equation}\label{trans.cop}
            \sum \tau(a\sw 1\ot a\sw 2) = \eps(a)1_H,
        \end{equation}
        \begin{equation}\label{trans.prod}
            \tau(a\ot b\cdot h)= \tau(a\ot b)h,
        \end{equation}
        \begin{equation}\label{trans.act}
            \sum a\sw 1\cdot \tau(a\sw 2\ot b) = \eps(a)b.
        \end{equation}
    \end{subequations}
    
The subspace 
\begin{equation}\label{coideal}
    I = \FF\langle a\ot b \eps(c) - \sum a\cdot \tau(b\ot c\sw 1) \ot c\sw 2\;|\; a,b,c\in C\rangle \subseteq C\ot C,
\end{equation}
is a coideal in $C^{\mathrm{co}}\ot C$. The coalgebra $\mathrm{E}(C,H) := C^{\mathrm{co}}\ot C/I$ is a Hopf algebra with identity, multiplication, and antipode
\begin{equation}\label{Ehr}
    1= \overline{\sum e\sw 1\ot e\sw 2 }, \quad \overline{a\ot b}\;\overline{c\ot d} = \overline{a\cdot \tau(b\ot c)\ot d}, \quad S(\overline{a\ot b}) = \overline{\sum a\cdot\tau(b\ot e\sw 1)\ot e\sw 2},
\end{equation}
where $e\in C$ is any element such that $\eps(e) =1$ and $\overline{a\ot b}$ indicates the class of $a\ot b \in C\ot C$ in $\mathrm{E}(C,H)$.

Similarly to \cite{Sch:bia} one obtains
\begin{lemma}\label{lem.Gal.Hopf}
 Let $H$ be a bialgebra and $C$ a right $H$-module colagebra satisfying conditions (a) and (b) of Definition~\ref{def.Gal}. Then $H$ is a Hopf algebra.  
\end{lemma}
\begin{proof}
    The proof dualises arguments of \cite{Sch:bia}. Let $e\in C$ be such that $\eps(e) =1$. Define
    \begin{equation}\label{antipode}
    S: H\to H, \qquad h\mapsto \sum \tau(e\sw 1\cdot h \otimes e\sw 2),    
    \end{equation}
    where $\tau$ is the cotranslation map. Then,
    $$
    \sum S(h\sw 1) h\sw 2 = \sum \tau(e\sw 1\cdot h\sw 1 \otimes e\sw 2)h\sw 2 = \sum \tau(e\sw 1\cdot h\sw 1 \otimes e\sw 2 \cdot h\sw 2) = \eps(h)1_H,
    $$
    by \eqref{trans.prod} and \eqref{trans.cop} combined with \eqref{mod.coa}.

    The equality $\sum h\sw 1 S(h\sw 2) = \eps(h)1_H$ is obtained by observing that the application of the isomorphism
    $$
    \Pi: \mathrm{Hom}(C\otimes H, H) \to \mathrm{Hom}^C(C\otimes H, C), \quad \Pi(f)(c\ot h) = \sum c\sw 1\cdot f(c\sw 2\ot h), $$
    where $\mathrm{Hom}^C(C\otimes H, C)$ denotes the space of all left $C$-comodule maps from $C\ot H$ to $C$, to the maps
    $$
    f(c\ot h) =\eps(c)\eps(h)1_H \;\;\; \& \;\;\; g(c\ot h) = \sum h\sw 1 \tau(c\sw 1\cdot h\sw 2\ot c\sw 2),
    $$
    yields an equality. We only note in passing that the inverse of $\Pi$ is given by 
    $$
    \Pi^{-1}(f)(c\ot h) = \sum \tau(c\sw 1 \ot f(c\sw 2\ot h)),$$
    for all $f\in \mathrm{Hom}^C(C\otimes H, C)$.
\end{proof}

\begin{theorem}\label{thm.Big}
 Let $(C,\chi)$ be a Hopf heap. Then:
 \begin{zlist}
     \item $C$ is a right Hopf-Galois co-object over the right translation Hopf algebra $\Tnr C$ with the action, for all $\tau_a^b\in \Tnr C$ and $c\in C$,
     $$
     c\cdot \tau_a^b = \tau_a^b(c) = [c,a,b].
     $$
     Furthermore, $\mathrm{E}(C,\Tnr C)\cong \Tnl C$.
     \item $C$ is a left Hopf-Galois co-object over the left translation Hopf algebra $\Tnl C$ with the action, for all $\sigma^a_b\in \Tnl C$ and $c\in C$,
     $$
     \sigma^a_b \cdot c = \sigma^a_b(c) = [a,b,c].
     $$
     \item $C$ is a $(\Tnl C,\Tnr C)$-bi-Galois co-object.
 \end{zlist}
\end{theorem}
\begin{proof}
    Since the action of $\Tnr C$ on $C$ is given by evaluation and the multiplication in $\Tnr C$ is given by the opposite composition $C$ is a right $\Tnr C$-module. Conditions \eqref{mod.coa} follow by the fact that the heap operation is a coalgebra map. Specifically and in particular, the first of \eqref{mod.coa} is an immediate consequence of \eqref{tau.coal} and the definition of the comultiplication in $\Tnr C$. 
    
    It is obvious that $\FF\langle c\cdot \tau_a^b - c\eps(\tau_a^b) \;|\; a,b,c\in C\rangle \subseteq \ker \eps$. Conversely, if $x\in \ker \eps$, then, for all $a\in C$ such that $\eps(a)=1$,
    $$
    \begin{aligned}
        x &= \eps(a)x - \eps(x)a = \sum \left(\eps(x\sw 2)\eps(a)x\sw 1 -  [x\sw 1,x\sw 2, a]\right) \\
        &= \sum\left(\eps(\tau_{x\sw 2}^a)x\sw 1 - x\sw 1 \cdot \tau_{x\sw 2}^a\right) ,
    \end{aligned}
    $$
    which proves the opposite inclusion. 

    The canonical map \eqref{can} is a linear isomorphism with the inverse 
    \begin{equation}\label{can.inv}
       \can^{-1}: C\ot C \to C\ot \Tnr C, \qquad a\ot b \mapsto \sum a\sw 1\ot \tau_{a\sw 2}^b. 
    \end{equation}
    Indeed, in one direction
    $$
    \begin{aligned}
        \can\circ\can^{-1}(a\ot b) = \sum a\sw 1\ot \tau_{a\sw 3}^b(a\sw 2) = \sum a\sw 1\ot [a\sw 2, a\sw 3, b] = a\ot b,
    \end{aligned}
    $$
    by \eqref{l.heap.c}, while in the other
    $$
    \begin{aligned}
        \can^{-1}\circ \can (c\ot \tau_a^b) &= \sum c\sw 1\ot \tau_{c\sw 2}^{\tau_a^b(c\sw 3)}
        = \sum c\sw 1 \ot \tau_{c\sw 2}^{[c\sw 3, a,b]}\\
        &= \sum c\sw 1 \ot \eps(c\sw 2)\tau_a^b = c\ot \tau_a^b, 
        \end{aligned}
    $$
    where the penultimate equality follows by \eqref{tau.theta}. Therefore, $C$ is a right Hopf-Galois co-object over $\Tnr C$. 

    In view of the form of the inverse of the canonical map \eqref{can.inv}, the cotranslation map comes out as
    \begin{equation}\label{Ehr.cot}
      \tau: C\ot C\to \Tnr C, \qquad a\ot b \mapsto \tau_a^b.
    \end{equation}
    Thus the coideal $I$ generating the Ehresmann-Schauenburg Hopf algebra $\mathrm{E}(C,\Tnr C)$ is
    $$
    I = \FF\langle a\ot b\eps(c) - \sum [a,b,c\sw 1]\ot c\sw 2 \; |\; a,b,c\in C\rangle.
    $$
    Consider the linear map
    \begin{equation}\label{Ehr.iso}
        \varphi : \mathrm{E}(C,\Tnr C) \to \Tnl C, \qquad \overline{a\ot b}\mapsto \sigma^a_b.
    \end{equation}
    The map $\varphi$ is well-defined, since, similarly to \eqref{tau.theta} one easily checks that, for all $a,b,c\in C$, 
    \begin{equation}\label{sig}
        \sum\sigma^{[a,b,c\sw 1]}_{c\sw 2} = \eps(c)\sigma^a_b,
    \end{equation}
    which immediately implies that for all $\sum_i a_i\ot b_i \in I$, 
$
    \sum_i \sigma^{a_i}_{b_i} =0.
    $

Clearly, $\varphi$ is a coalgebra map. By \eqref{sig}, for all $a\in C$, $\sum\sigma^{a\sw 1}_{a\sw 2} = \eps(a)\id $, hence $\varphi$ is unital. It is also multiplicative, since
$$
\varphi\left(\overline {a\ot b}\; \overline {c\ot d}\right) = \sigma^{[a,b,c]}_d = \sigma^a_b \sigma^c_d =\varphi\left(\overline {a\ot b}\right) \varphi\left(\overline{c\ot d}\right),
$$
by \eqref{left.str.a}, \eqref{Ehr} and \eqref{Ehr.cot}.

By construction, $\varphi$ is onto. It is also a monomorphism since $\sum_i\overline{a_i\ot b_i} \in \ker \varphi$ if and only if, for all $c\in C$, $\sum_i[a_i,b_i,c] = 0$. In particular, for any $c\in \ker\eps$,
$$
0 = \sum_i[a_i,b_i,c\sw 1] \ot c\sw 2= \sum_i[a_i,b_i,c\sw 1] \ot c\sw 2 - \sum_i a_i\ot b_i\eps(c),
$$
that is  $\sum_i\overline{a_i\ot b_i} =0$. 

In conclusion,  $\varphi$ is an isomorphism of bialgebras as required. 
    
    Statement (2) is proven by symmetric arguments or by invoking the fact that any  right $H$-Hopf-Galois co-object $C$ is an $(\mathrm{E}(C,H),H)$ bi-Galois co-object and using assertion (1). By the same token the statement (3) follows. We note only that the $(\Tnl C, \Tnr C)$-bimodule property follows by 
    \eqref{l.heap.a}, as for all $a,b,c,d,x \in C$
    $$
    (\sigma^a_b\cdot x)\cdot \tau_c^d = [[a,b,x],c,d] = [a,b,[x,c,d]] = \sigma^a_b\cdot (x\cdot \tau_c^d).
    $$
    This completes the proof of the theorem.
\end{proof}

\begin{corollary}\label{cor.trans.hopf}
Both $\Tnr C$ and $\Tnl C$ are Hopf algebras.
\end{corollary}
\begin{proof}
    This follows immediately for Theorem~\ref{thm.Big} and Lemma~\ref{lem.Gal.Hopf}. We only note that in view of \eqref{antipode}, the antipode in $\Tnr C$ comes out as 
   \begin{equation}\label{ant.tau}
       S(\tau_a^b) = \sum \tau_{[e\sw 1, a,b]}^{e\sw 2},
   \end{equation}
   for all $a,b\in C$, and $e\in C$  such that $\eps(e) =1$.

   Since the bialgebra $\Tnl C$ is isomorphic to the Hopf algebra $E(C,\Tnr C)$, it inherits an antipode via the isomorphism, thus becoming a Hopf algebra.
\end{proof}
\begin{theorem}\label{thm.Gal}
    Let $H$ be a Hopf algebra and $C$ be a right $H$-Hopf-Galois co-object. Then $C$ is a Hopf heap with the Grunspan map by the operation
  \begin{equation}\label{Gal.heap}
      \chi_{(C,H)}: C\ot C^{\mathrm{co}}\ot C\to C, \qquad a\ot b\ot c\mapsto a\cdot \tau(b\ot c),
  \end{equation}
  where $\tau$ is the cotranslation map \eqref{cotrans}. Furthermore, $H\cong \Tnr C$ as Hopf algebras.
\end{theorem}
\begin{proof}
    The theorem is a consequence of Theorem~\ref{thm.Big} and the results of Grunspan \cite{Gru:tor} and Schauenburg \cite{Sch:tor}, \cite{Sch:few}, but it can also be proven directly. 
    
Property \eqref{trans.prod} ensures that the condition \eqref{l.heap.a} for the operation \eqref{Gal.heap} holds. Equations \eqref{trans.cop} and \eqref{trans.act} yield the satisfaction of \eqref{l.heap.c}. The map $\chi_{(C,H)}$ is counital by \eqref{trans.counit}. That  it is also comultiplicative follows by the equality, for all $a,b\in C$,
\begin{equation}\label{trans.com}
\Delta(\tau(a\ot b)) = \sum \tau(a\sw 2\ot b\sw 1)\ot \tau(a\sw 1\ot b\sw 2).
\end{equation}
The proof of \eqref{trans.com} requires a bit of algebraic gymnastics. First, let us define the following map, which is a right $C$-coaction because $C$ is a right $H$-module coalgebra,
$$
\varrho: C\ot H \to C\ot H\ot C, \qquad a\ot h \mapsto \sum a\sw 1 \ot h\sw 1\ot a\sw 2\cdot h\sw 2.
$$
Then,
$$
(\id \ot \Delta)\circ \can = (\can \ot  \id)\circ \varrho,
$$
and so we obtain
\begin{equation}\label{can.rho}
    (\can^{-1}\ot \id) \circ (\id \ot \Delta) = \varrho\circ \can^{-1}.
\end{equation}
Next, observe that, for all $a,b\in C$,
\begin{equation}\label{can.trans}
    \can^{-1}(a\ot b) = \sum a\sw 1 \ot \tau (a\sw 2\ot  b).
\end{equation}
Combining \eqref{can.rho} with \eqref{can.trans} we arrive at
$$
\begin{aligned}
   \sum \tau (a\sw 2\ot b\sw 1)\ot  a\sw 1 \ot b\sw 2 &= \sum \tau(a\sw 3\ot b)\sw 1\ot a\sw1 \ot a\sw 2\cdot\tau (a\sw 3\ot b)\sw 2  \\
   &= \sum \tau(a\sw 2\ot b)\sw 1\ot \can \left(a\sw1 \ot \tau (a\sw 2\ot b)\sw 2\right)
\end{aligned}
$$
Equation \eqref{trans.com} now follows by appying  $\id \ot \tau$ to this equality.

Let us define the linear map
$$
\varphi_{(C,H)}: \Tnr C \to H, \qquad \tau_a^b \mapsto \tau (a\ot b).
$$
Note that this map is well-defined, since $\tau_a^b(c) = 0$, for all $c\in C$ if and only if $0=[c,a,b] = c\cdot \tau(a\ot b)$, for all $c\in C$. In particular, for all $c\in C$,
$$
0 = \sum c\sw 1 \ot c\sw 2\cdot \tau(a\ot b) = \can (c\ot \tau(a\ot b)),
$$
which implies that $\tau(a\ot b)=0$ for the canonical map is an isomorphism. 

The map $\varphi_{(C,H)}$ has the inverse,
$$
\varphi_{(C,H)}^{-1}: H\to \Tnr C, \qquad h\mapsto \tau_{e\sw 1}^{e\sw 2 \cdot h}\,, 
$$
where $e$ is any element of $C$ such that $\eps(e) =1$.  Indeed, that $\varphi_{(C,H)}\circ \varphi_{(C,H)}^{-1} = \id$ follows by \eqref{trans.prod} and \eqref{trans.cop}, while the other identity $\varphi_{(C,H)}^{-1}\circ \varphi_{(C,H)} =\id$ is a consequence of \eqref{tau.theta} in Lemma~\ref{lem.tau}.

The multiplicativity of $\varphi_{(C,H)}$ follows by \eqref{trans.prod}, since
$$
\varphi_{(C,H)}\left(\tau_a^b \tau_c^d\right) = \varphi_{(C,H)}\left(\tau_a^{[b,c,d]} \right) = \tau(a\ot b\cdot \tau(c\ot d)) = \tau(a\ot b)\tau(c\ot d).
$$
The unitality of $\varphi$ is a consequence of \eqref{tau.theta.a}  in Lemma~\ref{lem.tau} and \eqref{trans.cop}. Finally, $\varphi_{(C,H)}$ is a coalgebra map by \eqref{trans.com} (comultiplicativity) and \eqref{trans.counit} (counitality). Therefore, $\varphi_{(C,H)}$ is an isomorphism of Hopf algebras as required.

It remains to prove the existence of the Grunspan map. Before we work out the necessary form of this map from  \eqref{Grunspan.map}, we prove the following equality, satisfied by the cotranslation map:
\begin{equation}\label{cotran.s}
\tau(a\cdot \tau (b\ot c)\ot d) = S\tau(b\ot c)\tau(a\ot d),
\end{equation}
for all $a,b,c,d\in C$. First compute, for all $a\in C$ and $g,h\in H$,
$$
\sum \can(a\cdot g\sw 1\ot S(g\sw 2)h) = \sum a\sw 1\cdot g\sw 1\ot a\sw 2\cdot g\sw 2 S(g\sw 3)h = \sum a\sw 1\cdot g\ot a\sw 2 \cdot h.
$$
Applying $\tau$ to both sides of this we obtain
$$
\eps(a)S(g)h = \sum \tau(a\sw 1\cdot g\ot a\sw 2\cdot h).
$$
Setting $g=\tau(b\ot c)$ yields
$$
\eps(a) S\tau(b\ot c)h = \sum \tau(a\sw 1\cdot \tau(b\ot c)\ot a\sw 2\cdot h).
$$
Therefore, aaplying this equality to $\sum a\sw 1 \ot \tau(a\sw 2\ot d)$ instead of $a\otimes h$, we conclude
$$
\begin{aligned}
    S\tau(b\ot c)\tau(a\ot d) &= \sum \eps(a\sw 1) S\tau(b\ot c)\tau(a\sw 2\ot d)\\
    &= \sum \tau(a\sw 1\cdot \tau(b\ot c)\ot a\sw 2\cdot \tau(a\sw 3\ot d))\\
    &= \tau(a\cdot \tau(b\ot c)\ot d),
\end{aligned}
$$
where the last equality follows by \eqref{trans.act}. 

With \eqref{cotran.s}, \eqref{Grunspan.map} and \eqref{Gal.heap} at hand we can expect the following form for the Grunspan map:
\begin{equation}\label{Grun}
\vartheta: C\to C, \qquad c\mapsto \sum c\sw 1\cdot S\tau(c\sw 3\ot c\sw 2).
\end{equation}
Now it remains to check whether the property \eqref{assoc.theta} holds. 

We start by proving yet another property of the cotranslation map, namely that, for all $b,c\in C$,
\begin{equation}\label{cotrans.twist}
   \sum \tau(b\ot c\sw 1) S\tau(c\sw 3\ot c\sw 2) = S\tau(c\ot b). 
\end{equation}
To this end, let us consider the map
\begin{equation}\label{psi}
\psi: C\ot C\ot C\to H, \qquad b\ot c\ot a \mapsto \sum \tau(b\ot a\sw 1) S\tau(c\ot a\sw 2).
\end{equation}
Then, for all $h\in H$,
$$
\psi(b\ot c\ot a\cdot h) = \sum  \tau(b\ot a\sw 1) h\sw 1Sh\sw 2S\tau(c\ot a\sw 2) = \eps(h)\psi(b\ot c\ot a),
$$
by the fact that $C$ is a right $H$-module coalgebra and the property \eqref{trans.prod}. In view of condition (a) in Definition~\ref{def.Gal}, there exists map
$\bar{\psi}: C\ot C\to H$, such that, for all $a,b,c\in C$,
$$
\bar\psi(b\ot c\eps(a)) = \psi(b\ot c\ot a) = \sum \tau(b\ot a\sw 1) S\tau(c\ot a\sw 2).
$$
In particular,
$$
\bar\psi(b\ot c) = \sum \bar\psi(b\ot c\sw 2\eps(c\sw 1)) = \sum \tau(b\ot c\sw 1) S\tau(c\sw 3\ot c\sw 2)
$$
and 
$$
\bar\psi(b\ot c) = \sum \bar\psi(b\sw 1\ot c\eps(b\sw 2)) = \sum \tau(b\sw 1\ot b\sw 2) S\tau(c\ot b\sw 3) = S\tau(c\ot b),
$$
by \eqref{trans.com.cou},
and hence \eqref{cotrans.twist} follows. 

Finally, we can compute
$$
\begin{aligned}
    {}[[a,b,\vartheta(c)],d,e] &= \sum a\cdot \tau(b\ot c\sw 1)S\tau(c\sw 3\ot c\sw 2)\tau(d\ot e)\\
    &= a\cdot S\tau(c\ot b)\tau(d\ot e) = [a,[d,c,b],e],
\end{aligned}
$$
where the the first equality follows by \eqref{cotrans.twist}  and the second one by \eqref{cotran.s}. This completes the proof of the theorem.
\end{proof}

\begin{definition}\label{def.cat.Gal}
Let $(C,H)$ denote a right Hopf-Galois co-object $C$ over $H$ and $(D,H)$ a right Hopf-Galois co-object $D$ over $K$. A morphism from $(C,H)$ to $(D,K)$ is a pair of maps $(f,g)$ such that
\begin{blist}
    \item $f:C\to D$ is a homomorphism of coalgebras,
    \item $g:H\to K$ is a homomorphism of Hopf algebras,
    \item for all $c\in C$ and $h\in H$,
    \begin{equation}\label{cat.mor}
        f(c\cdot h) = f(c)\cdot g(h).
    \end{equation} 
\end{blist}
The category of right Hopf-Galois co-objects is denoted by $\mathcal{HG}$
\end{definition}

\begin{lemma}\label{lem.f.g}
If $(f,g)$ is a morphism of Hopf-Galois co-objects $(C,H)$ to $(D,K)$, then
\begin{equation}\label{f.g}
    \tau_D \circ (f\ot f) = g\circ \tau_C,
\end{equation}
where $\tau_C$ is the cotranslation map for $(C,H)$ and $\tau_D$ is the cotranslation map for $(D,K)$.    
\end{lemma}
\begin{proof}
    For all $c\in C$ and $h\in H$,
    $$
    \begin{aligned}
        \tau_D\circ (f\ot f)\circ \can_C(c\ot h) &= \tau_D\left(\sum f(c\sw 1)\ot f(c\sw 2 \cdot h)\right)\\
        &= \tau_D\left(\sum f(c)\sw 1\ot f(c)\sw 2 \cdot g(h)\right)\\
        &= \tau_D\circ \can_D (f(c)\ot g(h)) \\
        &= \eps(c)g(h) = g\circ \tau_C \circ \can_C(c\ot h),
    \end{aligned}
    $$
    since $f$ is a coalgebra map, by \eqref{cat.mor} and by the definition of the cotranslation map \eqref{cotrans}. The assertion follows by the bijectivity of the canonical map $\can_C$.
\end{proof}

In summary we obtain the following: 
\begin{theorem}\label{thm.equ}
    The functors
    $$
    \begin{aligned}
        \mathrm{Ga}: & \mathcal{HH}\to \mathcal{HG}, \qquad (C,\chi)\mapsto (C,\Tnr C), \quad f \mapsto (f,\Tnr f),\\
        \mathrm{He}: & \mathcal{HG}\to \mathcal{HH}, \qquad (C,H)\mapsto (C,\chi_{(C,H)}), \quad (f,g)\mapsto f,
    \end{aligned}
    $$
    are a pair of inverse equivalences between categories of Hopf heaps and right Hopf-Galois co-objects.
\end{theorem}
\begin{proof}
    Lemma~\ref{lem.f.g} ensures that $\mathrm{He}$ is a functor, specifically, if $(f,g)$ is a morphism of Hopf-Galois co-objects from $(C,H)\to (D,K)$, then $f\circ \chi_{(C,H)} = \chi_{(D,K)}\circ (f\ot f\ot f)$. One easily checks that $\chi_{(C,\Tnr C)} = \chi$, and hence $\mathrm{He}\circ \mathrm{Ga} = \id$. By Theorem~\ref{thm.Gal}, 
    $$
    \mathrm{Ga}\circ \mathrm{He}(C,H) = (C, \Tnr C) \cong (C, H),
    $$
    and so the required isomorphism of objects in $\mathcal{HH}$ is provided by the pair $(\id, \varphi_{(C,H)})$. This is a morphism in $\mathcal{HH}$ indeed, since, for all $a,b,c\in C$,
    $$
    a\cdot \tau_b^c = [a,b,c] = a\cdot \tau(b\ot c) = a\cdot \varphi_{(C,H)} \left(\tau_b^c\right).
    $$
    The naturality of this isomorphism, that is, the commutativity of the following diagram  in $\mathcal{HH}$
    $$
    \xymatrix{(C,\Tnr C) \ar[rr]^-{(\id,\varphi_{(C,H)})} \ar[d]_{(f, \Tnr f)} & & (C,H)\ar[d]^{(f,g)}\\
    D,\Tnr D) \ar[rr]^-{(\id,\varphi_{(D,K)})} && (D,K),}
    $$
    is  equivalent to 
    $$
    g\circ \varphi_{(C,H)} = \varphi_{(D,K)}\circ \Tnr f.
    $$
    Again, this follows by Lemma~\ref{lem.f.g}. Explicitly, for all $a,b\in C$,
    $$
    \begin{aligned}
       g\circ \varphi_{(C,H)}\left(\tau_a^b\right) &= g\left(\tau_C(a\ot b)\right) = \tau_D(f(a)\ot f(b))\\
       &= \varphi_{(D,K)} \left(\tau^{f(a)}_{f(b)}\right) = \varphi_{(D,K)}\circ \Tnr f\left(\tau_a^b\right).
    \end{aligned}
    $$
    This completes the proof of the theorem.
\end{proof}

Combining the discussion of the whole of the paper we obtain the following dual version of the main result of \cite{Sch:few}.
\begin{corollary}\label{cor.Grunspan}
    Every Hopf heap admits the Grunspan map.
\end{corollary}
\begin{proof}
    By Theorem~\ref{thm.trans} to any Hopf heap $(C,\chi)$ one can associate a bialgebra $\Tnr C$. Since it admits a Hopf-Galois co-object by  Theorem~\ref{thm.Big}, it is a Hopf algebra (see Corollary~\ref{cor.trans.hopf}).  Theorem~\ref{thm.Gal} ensures that the corresponding Hopf heap $(C, \chi_{(C,\Tnr C)})$ has the Grunspan map, and since $\chi_{(C,\Tnr C)} = \chi$ by (the proof of) Theorem~\ref{thm.equ}, the assertion follows. Explicitly, the Grunspan map is given by
    $$
    \vartheta: C\to C, \qquad c\mapsto \sum [c\sw 1, [e\sw 1, c\sw 3,c\sw 2],e\sw 2],
    $$
    where $e\in C$ is any element such that $\eps(e)=1$.
\end{proof}

\section*{Acknowledgements} 

This research is partially supported by the National Science Centre, Poland, grant no. 2019/35/B/ST1/01115.

\end{document}